\documentclass[a4paper,10pt,twoside]{article}
\usepackage{amsmath,amsfonts,amssymb}
\usepackage{graphicx}
\usepackage{eucal}
\usepackage{mathrsfs}
\usepackage{theorem}
\usepackage{pifont}
\usepackage {epsfig}
\usepackage {graphicx}
\newcommand{\Real}{\mathbb{R}}

\newcommand{\Natural}{\mathbb{N}}
\newcommand{\Complex}{\mathbb{C}}
\newcommand{\Disk}{\mathbb{D}}

\newcommand{\todo}[1]{{\sffamily To do:}}

\newtheorem{theorem}{Theorem}
\newtheorem{proposition}{Proposition}

\newtheorem {lemma}{Lemma}
\newenvironment{proof}{{\flushleft \emph{Proof}:}}{\ding{110}}

\newenvironment{proofgfull}{{\flushleft \emph{Proof of Lemma \ref{gfull}}:}}{\ding{110}}
\newenvironment{proofp00}{{\flushleft \emph{Proof of Proposition \ref{p00}}:}}{\ding{110}}

\numberwithin{equation}{section}

\title{Asymptotic behavior of random walks on a half-line with a jump at the origin}
\author{Guy Katriel\\Department of Mathematics, Ort Braude College,\\ Karmiel, Israel}

\date{}
\begin{document}

\maketitle

\begin{abstract} We study a discrete-time random walk on the non-negative integers, such that when $0$ is reached a jump occurs to an
arbitrary location $k\geq 0$ with probability $p_k$. We obtain an asymptotic formula for the expected position at time $n$, in dependence on the probabilities $p_k$ and on the starting position. Our proof of this result displays the relevance of the spectral analysis of the transition operator associated to the
stochastic process, both of its eigenvalues and of its resonances.
\end{abstract}

\section{Introduction}

We consider a discrete-time random walk performed on the non-negative integers $\{0,1,2,\cdots\}$ such that,
\begin{itemize}
  \item When at position $j\geq 1$, one moves left or right with equal probability,
  \item When at position $j=0$, one jumps to an arbitrary integer $k$, with probability $p_k$, where
\begin{equation}\label{sump}p_k\geq0,\;\;\;\;\sum_{k=0}^\infty p_k=1.\end{equation}
\end{itemize}
In other words, denoting by $X_n$ the position at time $n$:
\begin{eqnarray*}\label{ev}&&P(X_{n+1}=j-1\;|\;X_n=j)=P(X_{n+1}=j+1\;|\;X_n=j)=\frac{1}{2},\;\;\;j\geq 1,
\nonumber\\&&P(X_{n+1}=k\;|\;X_n=0)=p_k,\;\;\;k\geq 0.\end{eqnarray*}

This process is a special case of a more general class of processes considered by Lalley \cite{lalley}. Specialized to
the case considered here, the main result of \cite{lalley} is that, under the condition $p_0\leq \frac{1}{2}$, for each $j,k$ there exists a constant $C_{jk}$ so that, as $n\rightarrow \infty$,
$$P(X_n=k\;|\; X_0=j)\sim C_{jk}\frac{1}{\sqrt{n}},$$
where $\sim$ means that the quotient of the two sides approaches $1$.
Drmota (\cite{drmota}, Theorem 1) proved a more precise result in the special case in which $p_k=0$ for $k\geq 2$.

Here we are interested in obtaining a more precise characterization of the dependence of the asymptotics on the parameters defining the process, namely the
jump probabilities $p_k$ and the starting point $j$.
However, we shall not analyze the probabilities $P(X_n=k\;|\; X_0=j)$ themselves, but rather the expected position at time $n$, assuming the starting position $X_0=j$,
$$E(X_n\;|\; X_0=j)= \sum_{k=0}^\infty k P(X_n=k\;|\; X_0=j).$$
We prove the following asymptotic formula, as  $n\rightarrow \infty$
\begin{eqnarray}\label{cent}
&&E(X_n\;|\; X_0=j)=\frac{2}{\sqrt{2\pi}}\sqrt{n} +\frac{E(Y^2)-1}{2E(Y)}\\
&+&\frac{1}{\sqrt{2\pi}}\Big[\frac{1}{3}- \frac{E(Y^3)}{3E(Y)}+\frac{(E(Y^2)-1)^2}{2(E(Y))^2}-j\frac{E(Y^2)-1}{E(Y)}+j^2\Big]\frac{1}{\sqrt{n}}
  +O\Big(\frac{1}{n^{\frac{3}{2}}} \Big).\nonumber
\end{eqnarray}
where $Y$ is a random variable expressing the size of a jump, that is
$$E(Y=k)=p_k, \;\;k=0,1,2,\cdots$$
so that
$$E(Y^m)=\sum_{k=0}^\infty p_k k^m.$$

Several features of the above asymptotic formula are noteworthy:
\begin{itemize}
\item  The dominant term, of order $\sqrt{n}$, reflects the universality in the behavior of the process, as it does not depend on the jump probabilities $p_k$, nor on the
starting position $j$.

\item The second term, which is a constant shift independent of $n$, reflects the dependence of the process on the jump probabilities
$p_k$.  This term can be either positive or negative depending on whether $E(Y^2)>1$ or $E(Y^2)<1$.
Note that it too is independent of the starting position $j$.

\item The third term, of order $\frac{1}{\sqrt{n}}$, already depends on both the jumps at the origin and on the starting position $j$. One can thus say that
the memory of the initial position decays like $\frac{1}{\sqrt{n}}$.
\end{itemize}

The performance of formula (\ref{cent}) is illustrated in Table \ref{t1}, where the results given by the asymptotic formula (and its truncation to one or
two terms) are compared with the precise expected values, computed using the generating function derived in section \ref{gfunc}.

\begin{table}\label{t1}
\begin{center}
\begin{tabular}{|c|c|c|c|c|}
  \hline
 $n$ & $I$ & $I+II$ & $I+II+III$ & $E(X_n|X_0=j)$ \\
\hline
  10 & 2.52314 & 3.63425 & \underline{5}.40355 & 5.26359 \\
  20 & 3.56826 & 4.67937 & \underline{5}.93046 & 5.88291 \\
  50 & 5.64191 & 6.75302 & \underline{7.5}4430 & 7.53373 \\
  100 & 7.97885 & 9.08996 & \underline{9.64}944 & 9.64614 \\
  200 &  11.28379 & 12.39490 & \underline{12.7}9053 & 12.78946 \\
  400 &  15.95769 & 17.06880 & \underline{17.348}56 & 17.34820 \\
  \hline
\end{tabular}
\end{center}
\caption{Comparison of results from the asymptotic  formula of Theorem \ref{main} with the exact value of the expected position after $n$ steps, computed using the generating function. $I$,$II$,$III$ denote the three terms in the asymptotic formula. In this example $p_0=\frac{3}{10}, p_1=p_2=\frac{1}{10}, p_3=\frac{1}{2}$, and $j=5$.}
\end{table}

Formula (\ref{cent}) is proved under a few restrictions on the jump probabilities $\{p_k\}$, some of which are {\it{essential}}, in the sense that the
formula is not valid without them, and some {\it{technical}}, in that we believe they can be removed by elaborating the proof, but shall not do so here.
We now discuss these restrictions.

First of all, for (\ref{cent}) to make sense, we need that $E(Y)>0$. Clearly the only case in which $E(Y)=0$ is when $p_0=1$ and $p_k=0$ for $k\geq 1$, so that
when the origin is reached, one remains there at all times. In this exceptional case we will show (see section \ref{gfunc}) that the {\it{exact}} expected value is given by
the following somewhat surprising result:
\begin{proposition}\label{p00}
If $p_0=1$ then
$$E(X_n\;|\; X_0=j)=j,\;\;\forall j,n\geq 0.$$
\end{proposition}

Therefore, to exclude this special case, we make the assumption

(${\bf{A_1}}$)   $p_0<1.$

Another family of exceptional cases is when $p_k=0$ for all even $k$, so that the jump at $0$ is always to an odd value of $k$. In this case we will show that
the formula (\ref{cent}) is replaced by
\begin{eqnarray}\label{cent1}
&&E(X_n\;|\; X_0=j)=\frac{2}{\sqrt{2\pi}}\sqrt{n} +\frac{E(Y^2)-1}{2E(Y)}\\
&+&\frac{1}{\sqrt{2\pi}}\Big[\frac{1}{3}- \frac{E(Y^3)}{3E(Y)}+\frac{(E(Y^2)-1)^2}{2(E(Y))^2}-j\frac{E(Y^2)-1}{E(Y)}+j^2+\frac{1}{2}(-1)^{n+j+1}\Big]\frac{1}{\sqrt{n}}
  \nonumber\\&+&O\Big(\frac{1}{n^{\frac{3}{2}}} \Big)\nonumber
\end{eqnarray}
(note the added term in the coefficient of $\frac{1}{\sqrt{n}}$).
Therefore to exclude this special case, we also assume

(${\bf{A_2}}$)  There exists some even $k$ for which $p_k>0$.

Now we add two technical assumptions. First, our proof of (\ref{cent}) makes use of the assumption that

(${\bf{A_3}}$)  There exists an $N>0$ such that $p_k=0$ for $k>N$,

that is, the jump when the origin is reached can only be to a finite number of sites. We note that this ensures that the moments
$E(Y),E(Y^2),E(Y^3)$ which appear in (\ref{cent}) will all be finite, but it is natural to conjecture that (\ref{cent}) will remain valid
even when there are an infinite number of positive $p_k$'s, as long as these moments are finite. One may hope to prove (\ref{cent}) in the more
general case by approximating an infinite sequence of $p_k$'s by finite ones, but some estimates will be needed in order to control
the error term $O\Big(\frac{1}{n^{\frac{3}{2}}}\Big)$ during this approximation process, and we do not address this issue here.

To formulate our second technical assumption, we introduce the generating function for the probabilities $p_k$:
\begin{equation}\label{defh}
h(x)\doteq\sum_{k=0}^\infty p_k x^k,
\end{equation}
and the associated function
\begin{equation}\label{defphi}\phi(x)\doteq x^2+1-2xh(x),
\end{equation}
which will play an important role. By (\ref{sump}), the series in (\ref{defh}) converges at least for $|x|<1$, so that $h,\phi$  are holomorphic in
the unit disk. Under the assumption ($A_3$), $h$ and $\phi$ are in fact polynomials. We will use the assumption

(${\bf{A_4}}$) The polynomial $\phi$ does not have repeated roots.

Note that this is a `genericity' assumption, in the sense that for given degree, a generic polynomial of that degree will
satisfy it (the polynomials failing to satisfy it are those with discriminant $0$). This assumption can be eliminated by
approximating polynomials with repeated roots by generic polynomials with distinct roots, but this requires careful control of
error estimates, which we shall not carry out here. However, we do believe that the removal of assumption ($A_4$) is an easier matter than
the removal of assumption ($A_3$).

We can now state our main theorems

\begin{theorem}\label{main} Under assumptions ($A_1$)-($A_4$), we have (\ref{cent}) as $n\rightarrow \infty$.
\end{theorem}

\begin{theorem}\label{main2} Under  assumptions ($A_3$), ($A_4$) and

$({\bf{A_2^c}})$  $p_k=0$ for all even $k$,

we have (\ref{cent1}) as $n\rightarrow \infty$.
\end{theorem}

The proof of these theorems proceeds by the following steps
\begin{itemize}
\item Generating functions encoding the expected values are derived (section \ref{gfunc}). In this section we also explore the connection to the spectral analysis of the transition operator associated to the process.

\item In order to study the asymptotics of the coefficients in the power series expansion of the generating functions, these functions are
decomposed into sums of `simple' components (section \ref{decomposition}).

\item Asymptotics of the coefficients of the power series of the `simple' components are obtained using
classical results on the relations between singularities of analytic functions and the behavior of their Taylor coefficients (section \ref{asymptotics}).

\item Everything is combined to obtain Theorems \ref{main} and \ref{main2} (section \ref{mainproof}).
\end{itemize}

We note that the generating functions for the
probabilities, which easily imply the generating functions for the expected values, have already been derived in \cite{lalley}. However, we
give a complete derivation of these generating function in section \ref{gfunc}, because we use a different approach, which relies on functional-analytic rather
than probabilistic concepts. This approach reveals that the zeros of the function $\phi$ defined by (\ref{defphi}) are related to the eigenvalues and
resonances of the operator transition operator $L$ associated to our problem, and this fact puts the use of these zeros in the decomposition of the
generating function into simple components, which a key element of the derivation of the asymptotic formulas, into a more general perspective. Thus, while our final results, Theorems \ref{main},\ref{main2} contain no mention
of eigenvalues or resonances, and the use of the terminology of spectral theory is not essential for the proof of the of these result, the connection
with spectral concepts enriches our understanding.

Several directions for further work are suggested by the results and techniques presented here:

\begin{itemize}
\item It will be of interest to eliminate
the technical assumptions (A3),(A4) imposed here.

\item The asymptotic formulas are derived here by a technical calculation. It would be interesting to obtain a more intuitive understanding that
explains the particular form that the terms in the asymptotic formula take. For example, is there a way interpret the
constant shift term $\frac{E(Y^2)-1}{2E(Y)}$?

\item The decomposition technique developed here should be of use in treating other stochastic processes in which a simple process (such as the simple random walk) is locally perturbed (in our example by the jump at the origin).

\item While the relations between random walks (and other Markov processes) and eigenvalues is a huge topic (e.g. \cite{chen,lawler,woess}), it seems that the role of resonances in relation to stochastic processes has not received much attention. Resonances are most familiar to those studying quantum physics \cite{zworski}. It would be interesting to further explore the role of resonances in stochastic processes.
\end{itemize}

\section{Derivation of the generating functions}
\label{gfunc}

In this section we derive the explicit expressions for the generating functions of the expected values which are of interest to us.
The results in this section are valid for any sequence of probabilities $p_k$, and we do not need the assumptions ($A_1$)-($A_4$) introduced above.

For each $j\geq 0$, we define the generating function
\begin{equation}\label{defH}H_j(z)\doteq\sum_{n=0}^\infty E(X_n\;|\;X_0=j )z^n.\end{equation}

We will prove that
\begin{lemma}\label{gexp}
$$H_j(z)=2E(Y)\cdot\frac{1}{1-z}\cdot\frac{\rho(z)^{j+1}}{\phi(\rho(z))}+ j\frac{1}{1-z},$$
where
\begin{equation}\label{defrho}\rho(z)\doteq \frac{1-\sqrt{1-z^2}}{z}.
\end{equation}
and $\phi$ is defined by (\ref{defphi}).
\end{lemma}

Let us note that throughout this paper, the function $\sqrt{x}$ refers to the branch of the square root defined on $\Complex \setminus (-\infty,0)$ with $\sqrt{1}=1$.

Lemma \ref{gexp} will follow from an expression for the two-variable generating functions encoding the probabilities
\begin{equation}\label{defG}F_j(w,z)\doteq \sum_{n=0}^\infty \sum_{k=0}^\infty P(X_n=k|X_0=j)w^k z^n,\;\;j\geq 0.\end{equation}

\begin{lemma}\label{gfull}\begin{equation}\label{G}F_j(w,z)=\frac{2}{\phi(\rho(z))}\cdot\frac{(\rho(z))^{j+1}\phi(w)-\phi(\rho(z))w^{j+1}}{z(w^2+1)-2w}.\end{equation}
\end{lemma}

We note that equivalent expressions are already proved in \cite{lalley}.

To derive Lemma \ref{gexp} from Lemma \ref{gfull}, we note that, by differentiating (\ref{defG}) with respect to $w$, and then setting $w=1$, we obtain
\begin{eqnarray}\label{kk}\frac{\partial F_j}{\partial w}(1,z)&=&\sum_{j=0}^\infty   z^n\sum_{k=0}^\infty P(X_n=k\;|\;X_0=j )k\nonumber\\
&=& \sum_{n=0}^\infty  E(X_n\;|\; X_0=j) z^n =H_j(z)
.\end{eqnarray}
Differentiating (\ref{G}) with respect to $w$, we have
$$\frac{\partial F_j}{\partial w}(w,z)=\frac{2}{\phi(\rho(z))}\cdot\frac{(\rho(z))^{j+1}\phi'(w)-(j+1)\phi(\rho(z))w^{j}
}{z(w^2+1)-2w}$$
$$-\frac{2}{\phi(\rho(z))}\cdot\frac{[(\rho(z))^{j+1}\phi(w)-\phi(\rho(z))w^{j+1}][2wz-2]}{[z(w^2+1)-2w]^2}$$

and substituting $w=1$, taking into account that
$$\phi'(x)=2x-2xh(x)-2h'(x),$$
and that
$$h(1)=\sum_{k=0}^\infty p_k=1,\;\;\;h'(1)=\sum_{k=0}^\infty kp_k=E(Y),$$
so that
\begin{equation}\label{phi1}\phi(1)=2-2h(1)=0,\;\;\phi'(1)=-2h'(1)=-2E(Y),\end{equation}
we get, using (\ref{kk})
$$H_j(z)=\frac{\partial F_j}{\partial w}(1,z)=2E(Y)\frac{1}{1-z}\frac{(\rho(z))^{j+1}}{\phi(\rho(z))}+j\frac{1}{1-z},$$
which is the result of Lemma \ref{gexp}.

Let us note here that Lemma \ref{gexp} immediately gives the
\begin{proofp00}
If $p_0=1$ then $E(Y)=0$, so Lemma \ref{gexp} gives
$$H_j(z)=j\frac{1}{1-z},$$
that is, in view of the definition (\ref{defH}),  $E(X_n\;|\;X_0=j )=j$, for all $j,n$.
\end{proofp00}

We now give the

\begin{proofgfull}
We introduce the functions
$$f_{j,n}(w)=\sum_{k=0}^\infty P(X_n=k|X_0=j)w^k,$$
so that $F_j(w,z)$ defined by (\ref{defG}) can be written as
\begin{equation}\label{jon}F_j(w,z)=\sum_{n=0}^\infty f_{j,n}(w)z^n.\end{equation}
It will be useful to consider the functions $f_{j,n}(w)$ as elements of the Hilbert space $H^2$ of functions
analytic in the unit disk with square-summable coefficients, with the inner product
$$\langle g,h\rangle =\frac{1}{2\pi}\int_0^{2\pi} g(e^{i\theta})\overline{h(e^{i\theta})}d\theta.$$
We have
\begin{equation}\label{st}f_{j,0}(w)=w^j,\end{equation}
and, in view of (\ref{ev}), we can express $f_{j,n+1}$ in terms of $f_{j,n}$ through
\begin{equation}\label{it}f_{j,n+1}=L[f_{j,n}]\end{equation}
where $L:H^2\rightarrow H^2$ is the bounded linear operator defined through its action on the basis $\{w^j\}_{j=0}^\infty$ of $H^2$ by
$$L[w^j]=\frac{1}{2}w^{j-1}+\frac{1}{2}w^{j+1},\;\;j\geq 1,$$
$$L[1]=h(w)=\sum_{k=0}^\infty p_k w^k.$$
Let us note here that  $f_{j,n}$ encodes the state of the system started at $j$ at time $n$, since the coefficients of $f_{j,n}$ give probability of being in each place at that time.
By (\ref{it}), the operator $L$ describes the dynamics of the system, that is the transition from the state at time $n$ to the state at time $n+1$.
$L$ may then be called the transition operator, and it is quite natural that, as
we shall see below, the spectral analysis of $L$ is closely related to the behavior of the stochastic process.

The utility of formulating the operator in the Hilbert space $H^2$ is that $L$ can be represented in a simple form as
\begin{equation}\label{LL}L[f]=\frac{1}{2}(w^{-1}+w)[f(w)-f(0)]+f(0)h(w).\end{equation}
By (\ref{st}) and (\ref{it}) we have
\begin{equation}\label{bob}f_{j,n}(w)=L^n [f_{j,0}]=L^n [w^j].\end{equation}
If $|z|<\|L\|^{-1}$ then the resolvent
$$R(z)=[I-zL]^{-1}$$
is well defined, and we have the Neumann expansion
$$R(z)= \sum_{n=0}^\infty z^n L^n.$$

In view of (\ref{bob}) we have
\begin{equation}\label{repr}F_j(w,z)=\sum_{n=0}^\infty z^n L^n [w^j]=R(z)[w^j].\end{equation}
Therefore the function $F_j(w,z)$ can be computed by computing the resolvent $[I-zL]^{-1}$, which we will now do.

Let us then assume that $f,g\in H^2$ with $g\neq 0$ and
$$[I-zL]f=g$$
that is, by (\ref{LL}),
$$f(w)-z\frac{1}{2}(w^{-1}+w)[f(w)-f(0)]-zf(0)h(w)=g(w).$$
rearranging, we have
\begin{equation}\label{f11}f(w)=-\frac{2wg(w)+zf(0)[2wh(w)-(w^2+1)]}{z(w^2+1)-2w}=\frac{zf(0)\phi(w)-2wg(w)}{z(w^2+1)-2w},\end{equation}
where $\phi$ is given by (\ref{defphi}).
At this stage we have still not determined $f(w)$ completely, since the right-hand side of (\ref{f11}) contains $f(0)$, which still needs to be
determined. To do so, we observe that in order for the function $f$ defined by (\ref{f11}) to be an element of $H^2$, it is necessary that it will not have singularities in the unit disk.
We note that the quadratic equation
$$z(w^2+1)-2w=0$$
has two roots, given by
\begin{equation}\label{roots}w_1=\rho(z)=\frac{1-\sqrt{1-z^2}}{z}, \;\;w_2=\frac{1}{\rho(z)}=\frac{1+\sqrt{1-z^2}}{z},\end{equation}
and for $|z|<1$ it is easy to check that
\begin{equation}\label{lt1}|\rho(z)|<1.\end{equation}
We may therefore rewrite (\ref{f11}) in the form
\begin{equation}\label{f2}f(w)=\frac{zf(0)\phi(w)-2wg(w)}{z[w-\rho(z)][w-\frac{1}{\rho(z)}]}.\end{equation}
Since, by (\ref{lt1}), $w=\frac{1}{\rho(z)}$ is outside the unit disk, the only potential singularity is at $w=\rho(z)$, and in order for this not to be a singularity we must
have that the numerator vanishes at $w=\rho(z)$, that is we must require
\begin{equation}\label{ez}zf(0)\phi(\rho(z))-2\rho(z)g(\rho(z))=0.\end{equation}
Solved for $f(0)$, (\ref{ez}) gives
\begin{equation}\label{f0}f(0)=\frac{2\rho(z)g(\rho(z))}{z\phi(\rho(z))}.\end{equation}
Substituting (\ref{f0}) back into (\ref{f11}) we obtain
\begin{eqnarray}\label{res}R(z)[g]=f(w)=\frac{2}{\phi(\rho(z))}\cdot\frac{\rho(z)g(\rho(z))\phi(w)-\phi(\rho(z)) w g(w)}{z(w^2+1)-2w}.\end{eqnarray}
The function $f(w)$ defined the right hand side of (\ref{res}) is defined and holomorphic in $\Disk$ whenever
\begin{equation}\label{nez}\phi(\rho(z))\neq 0,\end{equation}
and, since $\rho(0)=0$ and $\phi(0)=1$, this condition holds for $|z|$ sufficiently small. If, in addition to (\ref{nez}), we assume that
\begin{equation}\label{nr}z \in \Complex \setminus \Big((-\infty,-1]\cup [1,\infty)\Big),\end{equation}
then we can check that $f\in H^2$, by using the integral form of the $H^2$ norm:
\begin{eqnarray}\label{h2c}&&\|f\|_{H^2}^2=\frac{1}{2\pi}\int_0^{2\pi}|f(e^{i\theta})|^2d\theta\\&=&\frac{1}{4\pi}\frac{1}{|\phi(\rho(z))|^2} \int_0^{2\pi} \Big|\frac{\rho(z)g(\rho(z))\phi(e^{i\theta})-\phi(\rho(z)) e^{i\theta} g(e^{i\theta})}{\cos(\theta)z-1} \Big|^2 d\theta\nonumber\\
&\leq& \frac{1}{4\pi}\frac{1}{|\phi(\rho(z))|^2}\cdot \max_{\theta\in [0,2\pi]}\frac{1}{|\cos(\theta)z-1|^2}\nonumber\\&\times& \int_0^{2\pi} \Big|\rho(z)g(\rho(z))\phi(e^{i\theta})-\phi(\rho(z)) e^{i\theta} g(e^{i\theta}) \Big|^2 d\theta,\nonumber
\end{eqnarray}
the maximum term being finite because of (\ref{nr}) and the integral being finite because of the assumption that $g\in H^2$, and the fact that
$\phi\in H^2$, which follows from (\ref{sump}). This estimate also shows that
$R(z)$ is bounded.

In particular, taking $g(w)=w^j$ in (\ref{res}), and using (\ref{repr}), we get
\begin{equation}\label{repp}F_j(w,z)=R(z)[w^j]=\frac{2}{\phi(\rho(z))}\frac{(\rho(z))^{j+1}\phi(w)-\phi(\rho(z))w^{j+1}}{z(w^2+1)-2w}.\end{equation}
\end{proofgfull}

We now make some comments regarding the spectral analysis of the transition operator $L$ introduced above.
We have seen above that the resolvent $R(z)=[I-z L]^{-1}$ exists when $z$ satisfies (\ref{nez}),(\ref{nr}). Since the spectrum $\sigma(L)$ of $L$ is
the set of values of $\frac{1}{z}$ such that $R(z)$ does not exist, we have thus shown that
$$\sigma(L)\subset \Big\{\frac{1}{z}\;\Big|\; \phi(\rho(z))=0\Big\} \cup [-1,1]. $$
In fact we can see that the above containment is an equality:

(i) The values $\frac{1}{z}$ for which
\begin{equation}\label{de} z\in \Complex\setminus \Big((-\infty,-1]\cup [1,\infty)\Big),\;\; \phi(\rho(z))=0,\end{equation}
are discrete eigenvalues of $L$: indeed, assuming (\ref{de}) and setting $g=0$  in (\ref{f2}), we have that, for any choice of $f(0)$, the function
$$f(w)=\frac{f(0)\phi(w)}{(w-\rho(z))(w-\frac{1}{\rho(z)})}$$
is holomorphic in $\Disk$ (note that $\phi(\rho(z))=0$ implies that the apparent singularity at $w=\rho(z)$ is removable),  $f\in H^2$ by
a calculation similar to that in (\ref{h2c}) (using the assumption $z\not\in (-\infty,-1]\cup [1,\infty)$),  and $[I-zL]f=0$.

(ii) The
shortest way to see that $[-1,1]\subset \sigma(L)$ is to note that $L$ is a compact (indeed a rank-one) perturbation of the discrete Laplacian, which is well known to have the essential spectrum $[-1,1]$, so that by Weyl's theorem the essential spectrum of $L$ is also $[-1,1]$. More directly, when $z\in (-\infty,-1]\cup [1,\infty)$ the
function $f$ defined by (\ref{res}) will have a pole of order $1$ on the boundary $\partial \Disk$, which will prevent $f$ from belonging to $H^2$.

We thus have the following decomposition of the spectrum of $L$ into the essential and discrete spectra:
\begin{eqnarray}\label{spec}&&\sigma(L)=\sigma_{ess}(L)\cup \sigma_{disc}(L)\\
&&\sigma_{ess}(L)=[-1,1]\nonumber\\
&&\sigma_{disc}(L)=\Big\{ \frac{1}{z}\;\Big|\; z\in \Complex \setminus(-\infty,-1]\cup [1,\infty),\;\; \rho(z)\in \phi^{-1}(0)\Big\}\nonumber
\end{eqnarray}
The function $\rho$ is a conformal mapping
\begin{equation}\label{conf}\rho: \Complex\setminus \Big((-\infty,-1]\cup [1,\infty)\Big)\rightarrow \Disk, \end{equation}
its inverse given by
\begin{equation}\label{inv}\rho^{-1}(\alpha)=\frac{2\alpha}{\alpha^2+1},\;\;\alpha\in \Disk.
\end{equation}
We thus see that those zeros of $\phi$ which are in the unit disc are in one-to-one correspondence with the eigenvalues of $L$.
We furthermore claim that the zeros of $\phi$ which are {\it{outside}} the unit disk correspond to the {\it{resonances}} of $L$ (see \cite{zworski} for an overview of resonances). To understand this,
consider the behavior of the expression on the right-hand side of (\ref{res}), giving the resolvent of $L$, as $z$ approaches the set $(-\infty,-1)\cup (1,\infty)$
from above or from below, that is as $\frac{1}{z}$ approaches the essential spectrum of $L$. Since the expression $1-z^2$ then approaches the segment
$(-\infty,0)$, where the square root has a branch cut, the limits as $z$ approaches $(-\infty,-1)\cup (1,\infty)$ from above or from below will be different,
but in each case analytic continuation is possible. Indeed, since the analytic continuation of $\sqrt{x}$ as $x$ crosses the negative axis is $-\sqrt{x}$,
the analytic continuation of $\rho(z)$ as $z$ crosses the set $(-\infty,-1)\cup (1,\infty)$ will be $\frac{1}{\rho(z)}$ (see (\ref{roots})), so that the expression
\begin{eqnarray}\label{ares}\tilde{R}(z)[g]=f(w)=\frac{2}{\phi(\frac{1}{\rho(z)})}\cdot\frac{\frac{1}{\rho(z)}g(\frac{1}{\rho(z)})\phi(w)-\phi(\frac{1}{\rho(z)}) w g(w)}{z(w^2+1)-2w}\end{eqnarray}
represents the meromorphic continuation of $R(z)$ across $(-\infty,-1)\cup (1,\infty)$ to the `second sheet'. Of course $\tilde{R}(z)$ is no longer a bounded operator
 from $H^2$ to itself, indeed the function $f(w)$ has a singularity at $w=\rho(z)\in \Disk$. The important point for us is that $\tilde{R}(z)$ has a pole
when $\phi(\frac{1}{\rho(z)})=0$, and since $\frac{1}{\rho(z)}$ maps $\Complex\setminus \Big((-\infty,-1]\cup [1,\infty)\Big)$ to the {\it{exterior}} of the unit disk,
we have that the set
$$res(L)=\Big\{ \frac{1}{z}\;\Big|\; \frac{1}{\rho(z)}\in \phi^{-1}(0)\Big\}$$
is now identified as the set of resonances of $L$, according to the definition of resonances as the
poles of the meromorphic continuation of the resolvent. We note that when $\rho(z)\in \partial\Disk\cap \phi^{-1}(0)$, then $\frac{1}{z}\in [-1,1]$ is a
resonance embedded in the essential spectrum. In fact $\frac{1}{z}=1$ is always such a resonance, since $\rho(1)=1\in \phi^{-1}(0)$ - see Lemma \ref{not1} below.

The zeros of $\phi$ will play an important role in the analysis carried out
in the following sections, aimed at obtaining the asymptotic formulas given by Theorems \ref{main},\ref{main2}, although these zeros do not appear in the
final formulas. The fact that these zeros corresponds to the eigenvalues and resonances of the operator $L$, as explained above, thus gives an
interesting perspective.

We shall now show that:
$$\sigma_{disc}(L)\subset \Disk.$$
In view of (\ref{spec}), this is equivalent to statement
\begin{lemma}\label{eid}
$$\phi(\rho(z))=0\;\;\Rightarrow \;\;|z|>1.$$
\end{lemma}

\begin{proof}  Assume that $\phi(\rho(z^*))= 0$. Set $\alpha^*=\rho(z^*)$. By (\ref{conf}) we have $|\alpha^*|<1$, and by (\ref{inv}) we have
\begin{equation}\label{inv1} \frac{2\alpha^*}{{\alpha^*}^2+1}=z^*.\end{equation}
Since $\phi(\alpha^*)=0$ we have, by (\ref{defphi}), that
\begin{equation}\label{uu}h(\alpha^*)=\frac{{\alpha^*}^2+1}{2\alpha^*}.\end{equation}
Combining (\ref{inv1}) and (\ref{uu}) we have $h(\alpha^*)=\frac{1}{z^*}$, hence
\begin{equation}\label{le}|z^*|=\frac{1}{|h(\alpha^*)|}.\end{equation}
Since $|\alpha^*|<1$, we have
$$|h(\alpha^*)|=\Big| \sum_{k=0}^\infty p_k {\alpha^*}^k  \Big|\leq  \sum_{k=0}^\infty p_k |\alpha^*|^k  < \sum_{k=0}^\infty p_k=1.$$
Together with (\ref{le}), we thus have $|z^*|>1$, as we wished to prove.
\end{proof}

An important immediate consequence of Lemma \ref{eid} is
\begin{lemma} The functions $H_j(z)$ are holomorphic in $\Disk$.
\end{lemma}

The following simple result will be used in later developments:

\begin{lemma}\label{not1} (i) $x=1$ is a zero of $\phi(x)$, and if ($A_1$) holds then it has multiplicity $1$.

(ii) $x=-1$ is a zero of $\phi(x)$ if and only if ($A_2^c$) holds.
\end{lemma}

\begin{proof}

(i) Since $h(1)=1$, we have $\phi(1)=2-2h(1)=0$. Since $\phi'(x)=2x-2h(x)-2xh'(x)$, we have
$$\phi'(1)=-2h'(1)=-2\sum_{k=0}^\infty kp_k=-2E(Y).$$
If ($A_1$) holds then $E(Y)>0$ and thus $\phi'(1)<0$, so that the root $x=1$ has multiplicity 1.

(ii) If ($A_2^c$) holds then $h(x)$ contains only odd powers of $x$, so $h(-1)=-1$, hence $\phi(-1)=0$. Conversely, if $\phi(-1)=0$ then $h(-1)=-1$, that is
$$\sum_{k=0}^\infty p_k (-1)^k=-1.$$
By (\ref{sump}), this can hold only if $p_k=0$ when $k$ is even, that is only when ($A_2^c$) holds.
\end{proof}

Assuming, as we shall henceforth do, that ($A_3$) holds, so that $h(x)$, hence also $\phi(x)$, are polynomials, we have,
denoting the degree of $h$ (that is the maximal value of $k$ for which $p_k\neq 0$) by $N$:

\begin{itemize}

\item If $N\geq 2$ then the degree of $\phi$ is $N$. The sum of the number of eigenvalues and resonances, counting multiplicities, is thus $N$.

\item If $N=1$ then $\phi(x)=x^2+1-2x[p_0+(1-p_0)x],$
so the degree of $\phi$ is $2$, except in the special case
\begin{equation}\label{spc}p_0=p_1=\frac{1}{2}\;\;\Leftrightarrow\;\;h(x)=\frac{1}{2}+\frac{1}{2}x\;\;\Leftrightarrow\;\; \phi(x)=1-x.\end{equation}
In this special case the only resonance is $z=1$, while otherwise there is an additional resonance given by $z=p_0+\frac{1-p_0}{2p_0-1}$.

\item If $N=0$ then $h(x)=1$, that is $p_0=1$, so that $\phi(x)=(x-1)^2$, so there is a double resonance at $z=1$. This case is
the one excluded by ($A_1$), and has already treated in proposition \ref{p00}.
\end{itemize}

\section{Decomposition of the generating functions}
\label{decomposition}

Given a function $f(z)$ holomorphic in a neighborhood of $0$, we use the standard notation $[z^n]f(z)$ to refer to the coefficient of $z^n$ in the power series
corresponding to $f$, that is $[z^n]f(z)=\frac{1}{n!}f^{(n)}(0)$.

From Lemma \ref{gexp} we have
\begin{equation}\label{ex} E(X_n\;|\;X_0=j )=[z^n][H_j(z)]=2E(Y)\cdot [z^n]\Big[\frac{1}{1-z}\cdot\frac{\rho(z)^{j+1}}{\phi(\rho(z))}\Big]+ j.\end{equation}
The problem we are faced with, then, is to estimate the power series coefficients on the right-hand side of (\ref{ex}).

By Lemma \ref{not1}(i), the
polynomial $\phi$ can be decomposed as
\begin{equation}\label{pd}\phi(x)=(x-1)\psi(x),\end{equation}
where $\psi(x)$ is a polynomial. Except in the special case given by (\ref{spc}), the degree of $\psi$ is at least $1$, so that it has
at least one zero. Since we shall use the zeros of $\psi$ below, we will now exclude the special case (\ref{spc}), and treat it by a separate calculation in section \ref{special}, which will show that Theorem \ref{main} is still valid in this case.

The first step is to perform a partial-fraction decomposition of the rational function $\frac{x}{\phi(x)}$.
At this point we are using the assumption ($A_4)$, so that the zeros of $\psi$ are non-degenerate, which allows us to decompose $\frac{1}{\psi(x)}$ as
\begin{equation}\label{pf}\frac{1}{\psi(x)}=\sum_{\psi(\alpha)=0}\frac{1}{\psi'(\alpha)}\frac{1}{x-\alpha},\end{equation}
and therefore
$$\frac{x}{\phi(x)}=\sum_{\psi(\alpha)=0}\frac{1}{\psi'(\alpha)}\frac{x}{x-1}\frac{1}{x-\alpha}=\sum_{\psi(\alpha)=0}\frac{1}{\psi'(\alpha)}\frac{1}{\alpha-1}\Big[ \frac{\alpha}{x-\alpha}-\frac{1}{x-1}\Big],$$
so that, substituting $x=\rho(z)$,
\begin{equation}\label{id1}\frac{1}{1-z}\cdot\frac{\rho(z)}{\phi(\rho(z))}=\sum_{\psi(\alpha)=0}\frac{1}{\psi'(\alpha)}\frac{1}{\alpha-1}\cdot\frac{1}{1-z}\Big[ \frac{\alpha}{\rho(z)-\alpha}-\frac{1}{\rho(z)-1}\Big].\end{equation}

Using the definition (\ref{defrho}) of $\rho(z)$, we have
$$\frac{1}{1-z}\cdot\frac{\alpha}{\rho(z)-\alpha}=\frac{1}{1-z}\cdot\frac{\alpha z}{1-\sqrt{1-z^2}-\alpha z}=\frac{\alpha}{1-z}\cdot\frac{1-\alpha z+\sqrt{1-z^2}}{(\alpha^2+1)z-2\alpha },$$
$$=\frac{\alpha}{(\alpha-1)^2}\Big[(\alpha^2+1)\frac{1-\alpha z+\sqrt{1-z^2}}{(\alpha^2+1)z-2\alpha}+\frac{1-\alpha z}{1-z}+\sqrt{\frac{1+z}{1-z}}\Big]$$
and
$$\frac{1}{1-z}\cdot \frac{1}{\rho(z)-1}=\frac{1}{1-z}\cdot \frac{z}{1-\sqrt{1-z^2}-z}=-\frac{1}{2}\cdot\Big[ \frac{1}{1-z}+ \frac{\sqrt{1+z}}{(1-z)^{\frac{3}{2}}}\Big],$$
so
$$\frac{1}{1-z}\Big[\frac{\alpha}{\rho(z)-\alpha}-\frac{1}{\rho(z)-1}\Big]$$
$$=\frac{\alpha}{(\alpha-1)^2}\frac{(\alpha-1)^2+(\alpha^2+1)[(1-\alpha)z+\sqrt{1-z^2}]}{(\alpha^2+1)z-2\alpha}
-\frac{1}{2}\frac{\alpha+1}{\alpha-1}\frac{1 }{1-z}$$
$$+\frac{\alpha}{(\alpha-1)^2}\sqrt{\frac{1+z}{1-z}}
+ \frac{1}{2}\frac{\sqrt{1+z}}{(1-z)^{\frac{3}{2}}}+\frac{\alpha}{\alpha-1},$$
so that (\ref{id1}) gives
$$\frac{1}{1-z}\frac{\rho(z)}{\phi(\rho(z))}$$
$$=\sum_{\psi(\alpha)=0}\frac{1}{\psi'(\alpha)}\Big[\frac{\alpha}{(\alpha-1)^3}\cdot\frac{(\alpha-1)^2+(\alpha^2+1)[(1-\alpha)z+\sqrt{1-z^2}]}{(\alpha^2+1)z-2\alpha}$$
$$+\frac{\alpha}{(\alpha-1)^3}\sqrt{\frac{1+z}{1-z}}-\frac{1}{2}\frac{\alpha+1}{(\alpha-1)^2}\frac{1}{1-z}+\frac{1}{2(\alpha-1)}\frac{\sqrt{1+z}}{(1-z)^{\frac{3}{2}}}
+\frac{\alpha}{(\alpha-1)^2}\Big].$$
We thus have the decomposition
\begin{eqnarray}\label{dec}H_0(z)&=&C_1 \sqrt{\frac{1+z}{1-z}}+C_2 \frac{1}{1-z}+C_3 \frac{\sqrt{1+z}}{(1-z)^{\frac{3}{2}}}+C_4\nonumber\\&+&2E(Y)\sum_{\psi(\alpha)=0}\frac{1}{\psi'(\alpha)}\frac{\alpha (\alpha^2+1)}{(\alpha-1)^3}K_{\alpha}(z),\end{eqnarray}
where
\begin{eqnarray}\label{cs}
C_1&\doteq&2E(Y)\sum_{\psi(\alpha)=0}\frac{1}{\psi'(\alpha)}\frac{\alpha}{(\alpha-1)^3},\nonumber\\
C_2&\doteq&-E(Y)\sum_{\psi(\alpha)=0}\frac{1}{\psi'(\alpha)}\frac{\alpha+1}{(\alpha-1)^2},\nonumber\\
C_3&\doteq& E(Y)\sum_{\psi(\alpha)=0}\frac{1}{\psi'(\alpha)}\frac{1}{\alpha-1},\nonumber\\
C_4&\doteq& 2E(Y)\sum_{\psi(\alpha)=0}\frac{1}{\psi'(\alpha)}\frac{\alpha}{(\alpha-1)^2},
\end{eqnarray}
and $K_{\alpha}(z)$ is given by
\begin{equation}\label{ka}K_{\alpha}(z)=\frac{(\alpha-1)^2+(\alpha^2+1)[(1-\alpha)z+\sqrt{1-z^2}]}{(\alpha^2+1)z-2\alpha}.\end{equation}

We now show that the numbers $C_1,C_2,C_3$ can be explicitly evaluated as follows ($C_4$ can be similarly evaluated, but it will not be needed):
\begin{lemma}\label{l4} $C_1,C_2,C_3$, defined by (\ref{cs}) are equal to
\begin{eqnarray*}C_1&=&-\frac{1}{12}- \frac{E(Y^3)}{6E(Y)}+\frac{(E(Y^2)-1)^2}{4(E(Y))^2},\nonumber\\
C_2&=&\frac{E(Y^2)-1}{2E(Y)},\nonumber\\
C_3&=&\frac{1}{2}.
\end{eqnarray*}
\end{lemma}

\begin{proof} Substituting $x=1$ into (\ref{pf}) we obtain
\begin{equation}\label{s1}\sum_{\psi(\alpha)=0}\frac{1}{\psi'(\alpha)}\frac{1}{1-\alpha}=\frac{1}{\psi(1)},\end{equation}
so
\begin{equation}\label{ss1}C_3=-E(Y)\frac{1}{\psi(1)}.\end{equation}
Differentiating (\ref{pf}) we get
\begin{equation}\label{pf1}\frac{\psi'(x)}{(\psi(x))^2}=\sum_{\psi(\alpha)=0}\frac{1}{\psi'(\alpha)}\frac{1}{(x-\alpha)^2},\end{equation}
and substituting $x=1$, we get
\begin{equation}\label{s2}\sum_{\psi(\alpha)=0}\frac{1}{\psi'(\alpha)}\frac{1}{(1-\alpha)^2}=\frac{\psi'(1)}{(\psi(1))^2}.\end{equation}
From (\ref{s1}),(\ref{s2}) and the decomposition $\frac{\alpha+1}{(1-\alpha)^2}=\frac{2}{(1-\alpha)^2}-\frac{1}{1-\alpha}$ we have
\begin{equation}\label{ss2}C_2=-E(Y)\sum_{\psi(\alpha)=0}\frac{1}{\psi'(\alpha)}\frac{\alpha+1}{(1-\alpha)^2}=-2E(Y)\frac{\psi'(1)}{(\psi(1))^2}+E(Y)\frac{1}{\psi(1)}.\end{equation}
Differentiating (\ref{pf1}) we have
\begin{equation*}\label{pf2}\frac{\psi''(x)\psi(x)-2(\psi'(x))^2}{(\psi(x))^3}=-2\sum_{\psi(\alpha)=0}\frac{1}{\psi'(\alpha)}\frac{1}{(x-\alpha)^3},\end{equation*}
and substituting $x=1$ we get
\begin{equation}\label{s3}\sum_{\psi(\alpha)=0}\frac{1}{\psi'(\alpha)}\frac{1}{(\alpha-1)^3}=\frac{\psi''(1)\psi(1)-2(\psi'(1))^2}{2(\psi(1))^3}.\end{equation}
From the decomposition $\frac{\alpha}{(\alpha-1)^3}=\frac{1}{(\alpha-1)^2}+\frac{1}{(\alpha-1)^3}$ and (\ref{s2}), (\ref{s3}) we get
\begin{eqnarray}\label{ss3}C_1&=&2E(Y)\sum_{\psi(\alpha)=0}\frac{1}{\psi'(\alpha)}\frac{\alpha}{(1-\alpha)^3}\\&=&E(Y)\cdot \frac{\psi''(1)\psi(1)-2(\psi'(1))^2+2\psi'(1)\psi(1)}{(\psi(1))^3}. \nonumber
\end{eqnarray}
Differentiating (\ref{pd}) we have
\begin{eqnarray}\label{pd1}\phi'(x)&=&\psi(x)+(x-1)\psi'(x),\;\;\;\phi''(x)=2\psi'(x)+(x-1)\psi''(x),\nonumber\\
\phi'''(x)&=&3\psi''(x)+(x-1)\psi'''(x),
\end{eqnarray}
and setting $x=1$ in (\ref{pd1}) we have
\begin{equation}\label{pd2}\psi(1)=\phi'(1),\;\;\psi'(1)=\frac{1}{2}\phi''(1),\;\;\psi''(1)=\frac{1}{3}\phi'''(1),\end{equation}
so that (\ref{ss1}),(\ref{ss2}),(\ref{ss3}) can be rewritten as
\begin{eqnarray}\label{sss1}C_1&=&E(Y)\Big[ \frac{\frac{1}{3}\phi'''(1)\phi'(1)-\frac{1}{2}(\phi''(1))^2+\phi''(1)\phi'(1)}{(\phi'(1))^3} \Big], \nonumber\\ C_2&=&E(Y)\Big(-\frac{\phi''(1)}{(\phi'(1))^2}+\frac{1}{\phi'(1)}\Big),\;\;\;C_3=-\frac{E(Y)}{\phi'(1)}.\end{eqnarray}
Differentiating (\ref{defphi}) we have
\begin{eqnarray}\label{defphi1}\phi'(x)&=& 2x-2h(x)-2xh'(x),\;\;\;\phi''(x)= 2-4h'(x)-2xh''(x),\nonumber\\
\phi'''(x)&=& -6h''(x)-2xh'''(x),
\end{eqnarray}
and substituting $x=1$ into (\ref{defphi1}), noting that
$$h(1)=\sum_{k=0}^N p_k=1,\;\;\;h'(1)=\sum_{k=0}^N k p_k=E(Y),\;\;\;$$
$$h''(1)=\sum_{k=0}^N k(k-1) p_k=E(Y^2)-E(Y),$$
$$h'''(1)=\sum_{k=0}^N k(k-1)(k-2) p_k=E(Y^3)-3E(Y^2)+2E(Y),$$
we have
\begin{eqnarray}\label{pd3}&&\phi'(1)=-2E(Y),\;\; \phi''(1)=2(1-E(Y)-E(Y^2)),\nonumber\\
&&\phi'''(1)=2(E(Y)-E(Y^3)).
\end{eqnarray}
Substituting (\ref{pd3}) into (\ref{sss1}), we get the results.
\end{proof}

\section{Asymptotics of some power series coefficients}
\label{asymptotics}

In this section we derive some results on the asymptotics of the coefficients of power series expansions of the basic components which were obtained
by decomposing the generating function $H_0(z)$ in section \ref{decomposition}.

A prototype result is (see \cite{flajolet}, Theorem VI.1)

\begin{lemma}\label{bas} For $\omega\in \Real\setminus \Natural$,
$$[z^n](1-z)^{\omega}=\frac{1}{\Gamma(-\omega)}\Big[ 1+\frac{\omega(\omega+1)}{2n}+O\Big(\frac{1}{n^2}\Big)\Big]\frac{1}{n^{\omega+1}}
,\;\;\;{\mbox{as}}\;\;n\rightarrow \infty.$$
\end{lemma}

We need to deal with more complicated functions which are holomorphic in the unit disk but have singularities of the type
$(1-z)^{\omega}$ on its boundary.

Let $f$ be a function holomorphic in the unit disk $\Disk$. A point $\xi$ with $|\xi|=1$ is called a {\it{singularity}} if there is no
neighborhood of $\xi$ such that $f$ can be continued analytically to this neighborhood. A singularity $\xi$ is called {\it{algebraic}} with weight $\omega$, where
$\omega\in \Real\setminus \Natural$, if
there exist holomorphic functions $A(z),B(z)$, defined in a neighborhood ${\cal{N}}$ of $\xi$, so that
\begin{equation}\label{rep}z\in \Disk\cap {\cal{N}}\;\Rightarrow\; f(z)=\Big(1-\frac{z}{\xi}\Big)^\omega A(z) +B(z).\end{equation}

The following is Darboux's Theorem (the form given here is taken from \cite{bender}):
\begin{lemma}\label{darboux}
Assume $f$ is holomorphic in $\Disk$ and has only a finite number of singularities on its boundary, all of them algebraic, and let the minimal weight of the
singularities be $\omega$. Denote the singularities with weight $\omega$ by  $\xi_1,\cdots, \xi_m$, and let $A_i(z),B_i(z)$ $(1\leq i\leq m)$ be holomorphic in a neighborhood of $\xi_i$ so that (\ref{rep}) holds with $\xi=\xi_i$.
 Then, as $n\rightarrow \infty$,
$$[z^n]f(z)=\Big[\frac{1}{\Gamma(-\omega)}\sum_{i=1}^m A_i(\xi_i)\Big]\frac{1}{n^{\omega+1}} +o\Big(\frac{1}{n^{\omega+1}} \Big),$$
\end{lemma}

An immediate corollary of Lemma \ref{darboux} is
\begin{lemma}\label{gene} If $j,k$ are integers with $j>k$, and $k$ is odd, then, as $n\rightarrow \infty$,
$$[z^n]\Big[(1+z)^{\frac{j}{2}} (1-z)^{\frac{k}{2}}\Big]= \frac{2^{\frac{j}{2}}}{\Gamma(-\frac{k}{2})}\frac{1}{n^{\frac{k}{2}+1}}+o\Big(\frac{1}{n^{\frac{k}{2}+1}} \Big),$$
\end{lemma}

\begin{proof}
$\xi=1$ is the singular point with minimal weight, $\omega=\frac{k}{2}$, $A(z)=(1+z)^{\frac{j}{2}}$, $B(z)=0$. Hence the result follows from Lemma \ref{darboux}.
\end{proof}

We now derive the estimates that we need for obtaining our main results.

\begin{lemma}\label{l2}  As $n\rightarrow \infty$,
$$[z^n]\Big[\sqrt{\frac{1+z}{1-z}}(1-z^2)^{\frac{l}{2}}\Big]=\left\{
                                                               \begin{array}{ll}
                                                                 \frac{2}{\sqrt{2\pi}}\frac{1}{\sqrt{n}}+O\Big(\frac{1}{n^{\frac{3}{2}}}\Big), & l=0 \\
                                                                 O(\frac{1}{n^{\frac{3}{2}}}), & l\geq 1.
                                                               \end{array}
                                                             \right.$$
\end{lemma}

\begin{proof} Set
$$f(z)=\sqrt{\frac{1+z}{1-z}}(1-z^2)^{\frac{l}{2}}=(1+z)^{\frac{l+1}{2}}(1-z)^{\frac{l-1}{2}}.$$
If $l\geq 1$ is odd then this function is a polynomial, so that its coefficients vanish for large $n$, so the $O\Big(\frac{1}{n^{\frac{3}{2}}}\Big)$ estimate is
trivially true.

We now consider the case that $l\geq 2$ is even. Then Lemma \ref{gene}, with $j=l+1$, $k=l-1$ gives
$$[z^n]f(z)= \frac{2^{\frac{l+1}{2}}}{\Gamma(-\frac{l-1}{2})}\frac{1}{n^{\frac{l+1}{2}}}+o\Big(\frac{1}{n^{\frac{l+1}{2}}} \Big)=O\Big(\frac{1}{n^{\frac{3}{2}}}\Big).$$
There remains the case $l=0$.
We write
\begin{equation}\label{dd}f(z)=\sqrt{\frac{1+z}{1-z}}=g(z)+\sqrt{2}\frac{1}{\sqrt{1-z}},\end{equation}
where
$$g(z)=\frac{k(z)}{\sqrt{1-z}},\;\;k(z)=\sqrt{1+z}-\sqrt{2}$$
so that $k(1)=0$, implying that
$$g(z)=A(z)(1-z)^{\frac{1}{2}}$$
where
$$A(z)=\frac{k(z)}{1-z}$$
is holomorphic in a neighborhood of $z=1$. Thus $g$ has an algebraic singularity with weight $\frac{1}{2}$ at $z=1$. Clearly it also has
an algebraic singularity with weight $\frac{1}{2}$ at $z=-1$. From Lemma \ref{darboux} we therefore have
\begin{equation}\label{dd1}[z^n]g(z)=O\Big(\frac{1}{n^{\frac{3}{2}}}\Big).\end{equation}
By Lemma \ref{bas} we have
\begin{equation}\label{dd3}[z^n]\Big(\frac{1}{\sqrt{1-z}}\Big)=\frac{1}{\sqrt{\pi}}\frac{1}{n^{\frac{1}{2}}}\Big( 1+O\Big(\frac{1}{n}\Big)\Big).\end{equation}
Combining (\ref{dd}),(\ref{dd1}) and (\ref{dd3}) we obtain
$$[z^n]f(z)=\frac{\sqrt{2}}{\sqrt{\pi}}\frac{1}{\sqrt{n}}+O\Big(\frac{1}{n^{\frac{3}{2}}}\Big).$$
\end{proof}

\begin{lemma}\label{l5}  As $n\rightarrow \infty$,
$$[z^n]\Big[\frac{(1-z^2)^{\frac{l}{2}}}{1-z} \Big]=\left\{
                                                      \begin{array}{ll}
                                                        1, & l=0 \\
                                                        \frac{2}{\sqrt{2\pi}}\frac{1}{\sqrt{n}}+O\Big(\frac{1}{n^{\frac{3}{2}}}\Big) , & l=1 \\
                                                        O\Big(\frac{1}{n^{\frac{3}{2}}}\Big), & l\geq 2.
                                                      \end{array}
                                                    \right.
$$
\end{lemma}

\begin{proof}
Set
$$f(z)=\frac{(1-z^2)^{\frac{l}{2}}}{1-z}=(1+z)^{\frac{l}{2}}(1-z)^{\frac{l-2}{2}}.$$
When $l=0$ we have
$$[z^n]f(z)=[z^n]\Big(\frac{1}{1-z}\Big)=1.$$
When $l\geq 2$ is even then $f(z)$ is a polynomial of degree $j$, so the coefficients vanish for large $n$, and the $O\Big(\frac{1}{n^{\frac{3}{2}}}\Big)$ estimate is
trivially true.

Assuming $l$ is odd, Lemma \ref{gene}, with $j=l$, $k=l-2$ gives
$$[z^n]f(z)= \frac{2^{\frac{l}{2}}}{\Gamma(1-\frac{l}{2})}\frac{1}{n^{\frac{l}{2}}}+o\Big(\frac{1}{n^{\frac{l}{2}}} \Big).$$
In the case $l\geq 3$, this gives the required estimate.

If $l=1$ then $f(z)=\sqrt{\frac{1+z}{1-z}}$, and the required estimate follows from the case $l=0$ of Lemma \ref{l2}.
\end{proof}

\begin{lemma}\label{l3}  As $n\rightarrow \infty$,
$$[z^n]\Big[\frac{\sqrt{1+z}}{(1-z)^{\frac{3}{2}}}(1-z^2)^{\frac{l}{2}}\Big]=\left\{
                                                                               \begin{array}{ll}
                                                                                 \frac{4}{\sqrt{2\pi}} \sqrt{n}+\frac{1}{\sqrt{2\pi}}\frac{1}{\sqrt{n}}+O\Big(\frac{1}{n^{\frac{3}{2}}}\Big), & l=0 \\
                                                                                 2+O\Big(\frac{1}{n^{\frac{3}{2}}}\Big), & l=1 \\
                                                                     \frac{4}{\sqrt{2\pi}}\frac{1}{\sqrt{n}}+O\Big(\frac{1}{n^{\frac{3}{2}}}\Big), & l=2 \\
                                                                                 O\Big(\frac{1}{n^{\frac{3}{2}}}\Big), & l\geq 3.
                                                                               \end{array}
                                                                             \right.
$$
\end{lemma}

\begin{proof} Set
$$f(z)=\frac{\sqrt{1+z}}{(1-z)^{\frac{3}{2}}}(1-z^2)^{\frac{l}{2}}=(1+z)^{\frac{l+1}{2}}(1-z)^{\frac{l-3}{2}}.$$
If $l\geq 3$ is odd then $f(z)$ is a polynomial, so the coefficients vanish for large $n$, and the estimate holds trivially.

If $l$ is even then from Lemma \ref{gene} with $j=l+1$, $k=l-3$ it follows that
$$[z^n]f(z)= \frac{2^{\frac{l+1}{2}}}{\Gamma(-\frac{l-3}{2})}\frac{1}{n^{\frac{l-1}{2}}}+o\Big(\frac{1}{n^{\frac{l-1}{2}}} \Big).$$
Thus if $l\geq 4$ then
$$[z^n]f(z)=O\Big(\frac{1}{n^{\frac{3}{2}}}\Big).$$
If $l=2$ we decompose
$$f(z)=g(z)+2^{\frac{3}{2}}\frac{1}{\sqrt{1-z}}$$
where
$$g(z)=\frac{k(z)}{\sqrt{1-z}},\;\;\;k(z)=(1+z)^{\frac{3}{2}}-2^{\frac{3}{2}},$$
and since $k(1)=0$ we have
$$g(z)=A(z)\sqrt{1-z},\;\;A(z)=\frac{k(z)}{1-z},$$
where $A(z)$ is holomorphic in a neighborhood of $z=1$, so that $z=1$ is an algebraic singularity of $g$ with weight $\frac{1}{2}$. $z=-1$
is also an algebraic singularity, with weight $\frac{3}{2}$. Therefore Lemma \ref{darboux} implies that
$$[z^n]g(z)=O\Big(\frac{1}{n^{\frac{3}{2}}}\Big),$$
and using Lemma \ref{bas} we have
$$[z^n]f(z)=[z^n]g(z)+2^{\frac{3}{2}}[z^n]\Big[\frac{1}{\sqrt{1-z}} \Big]=\frac{2^{\frac{3}{2}}}{\sqrt{\pi}}\frac{1}{\sqrt{n}}+O\Big(\frac{1}{n^{\frac{3}{2}}}\Big).$$
If $l=1$ then
$$f(z)=\frac{1+z}{1-z}$$
so
$$[z^n]f(z)=2,\;\;n\geq 1.$$
If $l=0$ we decompose
\begin{equation}\label{deci}f(z)=\frac{\sqrt{1+z}}{(1-z)^{\frac{3}{2}}}=g(z)
+\frac{\sqrt{2}}{(1-z)^{\frac{3}{2}}}-\frac{\sqrt{2}}{4}
\frac{1}{\sqrt{1-z}},\end{equation}
where
$$g(z)=\frac{k(z)}{(1-z)^{\frac{3}{2}}},\;\;k(z)=\sqrt{1+z}-\sqrt{2}+\frac{\sqrt{2}}{4}(1-z)$$
is chosen so that $k(1)=k'(1)=0$, hence setting
$$g(z)=A(z)\sqrt{1-z},\;\;A(z)=\frac{k(z)}{(1-z)^2},$$
we have that $A(z)$ is holomorphic in a neighborhood of $z=1$, so that $z=1$ is an algebraic singularity of $g$ with weight $\frac{1}{2}$. $z=-1$ is also an
algebraic singularity with weight $\frac{1}{2}$. Therefore lemma \ref{darboux} implies that
$$[z^n]g(z)=O\Big(\frac{1}{n^{\frac{3}{2}}}\Big).$$
Therefore, using (\ref{deci}) and Lemma \ref{bas}
$$[z^n]f(z)=[z^n]g(z)+\sqrt{2}[z^n]\Big[\frac{1}{(1-z)^{\frac{3}{2}}}\Big]-\frac{\sqrt{2}}{4}[z^n]\Big[
\frac{1}{\sqrt{1-z}}\Big]$$
$$=O\Big(\frac{1}{n^{\frac{3}{2}}} \Big)+\sqrt{2}\frac{1}{\Gamma(\frac{3}{2})}\Big( 1+\frac{3}{8n}+O\Big(\frac{1}{n^2}\Big)\Big)n^{\frac{1}{2}}-\frac{\sqrt{2}}{4}\frac{1}{\Gamma(\frac{1}{2})}\Big( 1+O\Big(\frac{1}{n}\Big)\Big)\frac{1}{n^{\frac{1}{2}}}$$
$$=\frac{4}{\sqrt{2\pi}} \sqrt{n}+\frac{1}{\sqrt{2\pi}}\frac{1}{\sqrt{n}}+O\Big(\frac{1}{n^{\frac{3}{2}}} \Big).$$
\end{proof}

\begin{lemma}\label{lll}
For any $l\geq 0$, as $n\rightarrow \infty$,
$$[z^n][(1-z^2)^{\frac{l}{2}}]=O\Big(\frac{1}{n^{\frac{3}{2}}}\Big).$$
\end{lemma}

\begin{proof} For $l$ even the function is a polynomial, so the result holds trivially. For $l$ odd the result follows from Lemma \ref{darboux}.
\end{proof}

\begin{lemma}\label{l1} For $\alpha\in \Complex$, let $K_{\alpha}(z)$ be defined by (\ref{ka}).

(i) For any $\alpha\neq \pm 1$ with $\phi(\alpha)=0$ we have, for any integer $l\geq 0$, as $n\rightarrow\infty$,
$$[z^n]\Big[K_{\alpha}(z)(1-z^2)^{\frac{l}{2}}\Big]=O\Big(\frac{1}{n^{\frac{3}{2}}} \Big).$$

(ii) For $\alpha=-1$ we have, as $n\rightarrow\infty$,
$$[z^n]\Big[K_{-1}(z)(1-z^2)^{\frac{l}{2}}\Big]=\left\{
                                                               \begin{array}{ll}
                                                                 (-1)^n\frac{2}{\sqrt{2\pi}}\frac{1}{\sqrt{n}}+O\Big(\frac{1}{n^{\frac{3}{2}}} \Big), & l=0 \\
                                                                 O\Big(\frac{1}{n^{\frac{3}{2}}} \Big), & l\geq 1.
                                                               \end{array}
                                                             \right.$$
\end{lemma}

\begin{proof} Set:
$$f(z)=K_{\alpha}(z)(1-z^2)^{\frac{l}{2}}=\frac{(\alpha-1)^2+(\alpha^2+1)[(1-\alpha)z+\sqrt{1-z^2}]}{(\alpha^2+1)z-2\alpha}(1-z^2)^{\frac{l}{2}}.$$
We first claim that $K_\alpha(z)$, and hence also $f(z)$, is holomorphic in $\Disk$. Since the numerator and denominator are
obviously holomorphic on this set, it is only necessary to show
that the apparent singular point when the denominator vanishes, that is
\begin{equation}\label{z0}z_0=\frac{2\alpha}{\alpha^2+1},\end{equation}
either satisfies $|z_0|>1$, or is
is removable. There are three cases to consider:

(i) If $|\alpha|<1$, then from (\ref{inv}) and (\ref{z0}) we have that $\rho(z_0)=\alpha$, and thus by the assumption that $\phi(\alpha)=0$ and Lemma
\ref{eid} we have that $|z_0|>1$, so that the singularity is outside the unit disk.

(ii) If $|\alpha|=1$ then $z_0\in (-\infty,-1] \cup [1,\infty)$, so the singularity is outside the unit disk.

(iii) If $|\alpha|>1$ then we will show that the numerator of $K_\alpha(z)$ vanishes at $z=z_0$, so that the apparent singularity at $z=z_0$ is removable. Indeed,
$|\alpha|>1$ implies that
$$\Re\Big(\frac{\alpha^2-1}{\alpha^2+1}\Big)>0,$$
so that
$$\sqrt{\Big(\frac{\alpha^2-1}{\alpha^2+1}\Big)^2}=\frac{\alpha^2-1}{\alpha^2+1}$$
(note that the equality $\sqrt{x^2}=x$ holds only when $\Re(x)>0$).
Therefore, substituting $z=z_0$ into the numerator of $K_{\alpha}(z)$ we obtain
$$(\alpha-1)^2+(\alpha^2+1)\Big[(1-\alpha) z_0+\sqrt{1-z_0^2}\Big]$$
$$=(\alpha-1)^2+(\alpha^2+1)\Big[(1-\alpha) \frac{2\alpha}{\alpha^2+1}+\sqrt{\Big(\frac{\alpha^2-1}{\alpha^2+1}\Big)^2}\Big]=0.$$

We now assume that $\alpha\neq \pm 1$.

The singularities of $f$ on the boundary of the unit disk are $\xi_\pm=\pm 1$, both algebraic singularities.

If $l$ is odd then the weight of both singularities is $\omega=\frac{l}{2}$, since $f$ can be represented as
$$f(z)=A_+(z)(1-z)^{\frac{l}{2}}+B(z),\;\;f(z)=A_-(z)(1+z)^{\frac{l}{2}}+B(z),$$
where
$$A_{\pm}(z)=\frac{[(\alpha-1)^2+(\alpha^2+1)(1-\alpha) z](1\pm z)^{\frac{l}{2}}}{(\alpha^2+1)z-2\alpha},$$
$$B(z)=\frac{\alpha^2+1}{(\alpha^2+1)z-2\alpha}(1-z^2)^{\frac{l+1}{2}}$$
are holomorphic in neighborhoods of $\pm 1$, respectively. Note that here we use the assumption $\alpha\neq \pm 1$, and that to conclude that $B$ is holomorphic in
 neighborhoods of $z=\pm 1$ we use the assumption that $l$ is odd.

Lemma \ref{darboux} thus gives
$$[z^n]f(z)=C\frac{1}{n^{\frac{l}{2}+1}} +o\Big( \frac{1}{n^{\frac{l}{2}+1}}\Big)=O\Big(\frac{1}{n^{\frac{3}{2}}} \Big).$$

If $l$ is even then the singularities $\xi_\pm=\pm 1$ have weights $\omega=\frac{l+1}{2}$, since $f$ can be represented as
$$f(z)=A_+(z)(1-z)^{\frac{l+1}{2}}+B(z),\;\;f(z)=A_-(z)(1+z)^{\frac{l+1}{2}}+B(z),$$
$$A_\pm(z)=\frac{(\alpha^2+1)(1\pm z)^{\frac{l+1}{2}}}{(\alpha^2+1)z-2\alpha},$$
$$B(z)=\frac{(\alpha-1)^2+(\alpha^2+1)(1-\alpha) z}{(\alpha^2+1)z-2\alpha}(1-z^2)^{\frac{l}{2}}.$$
Therefore lemma \ref{darboux} gives
$$[z^n]f(z)=C\frac{1}{n^{\frac{l}{2}+\frac{3}{2}}} +o\Big( \frac{1}{n^{\frac{l}{2}+\frac{3}{2}}}\Big)=O\Big(\frac{1}{n^{\frac{3}{2}}} \Big).$$

There remains the case $\alpha=-1$. In this case we have
$$f(z)=\Big[2+ \sqrt{\frac{1-z}{1+z}}\Big](1-z^2)^{\frac{l}{2}}.$$
Using Lemma \ref{l2} we have
$$[z^n]\Big[\sqrt{\frac{1-z}{1+z}}(1-z^2)^{\frac{l}{2}} \Big]=[z^n]\Big[(1-(-z)^2)^{\frac{l}{2}}\sqrt{\frac{1+(-z)}{1-(-z)}} \Big]$$
$$=(-1)^n[z^n]\Big[\sqrt{\frac{1-z}{1+z}} (1-z^2)^{\frac{l}{2}}\Big]$$
$$=\left\{
                                                               \begin{array}{ll}
                                                                 (-1)^n\frac{2}{\sqrt{2\pi}}\frac{1}{\sqrt{n}}+O\Big(\frac{1}{n^{\frac{3}{2}}}\Big), & l=0 \\
                                                                 O(\frac{1}{n^{\frac{3}{2}}}), & l\geq 1.
                                                               \end{array}
                                                             \right.$$
From Lemma \ref{bas} we have, for $l\geq 1$ odd,
$$[z^n](1-z^2)^{\frac{l}{2}}=\frac{1}{\Gamma(-\frac{l}{2})}\Big[ 1+O\Big(\frac{1}{n}\Big)\Big]\frac{1}{n^{\frac{l}{2}+1}}=O\Big(\frac{1}{n^{\frac{3}{2}}}\Big),$$
and the same estimate trivially holds for even $l$ because the function is then a polynomial.

Therefore if $l\geq 1$ we have $[z^n]f(z)=O(\frac{1}{n^{\frac{3}{2}}})$, while for $l=0$ we get
$$[z^n]f(z)=(-1)^n\frac{2}{\sqrt{2\pi}}\frac{1}{\sqrt{n}}+O\Big(\frac{1}{n^{\frac{3}{2}}}\Big).$$
\end{proof}

\section{Proof of the main theorems}
\label{mainproof}

We now have the tools needed to finish the proofs of Theorems \ref{main} and \ref{main2}.

We consider first the case $j=0$ of Theorems \ref{main},\ref{main2}. From (\ref{ex}),(\ref{dec})  we have, for $n\geq 1$:
\begin{eqnarray}\label{calc}E(X_n\;|\;X_0=0)&=&[z^n][H_0(z)]=C_1 [z^n]\Big[\sqrt{\frac{1+z}{1-z}}\Big]+C_2 +C_3 [z^n]\Big[\frac{\sqrt{1+z}}{(1-z)^{\frac{3}{2}}}\Big]\nonumber\\
&+&2E(Y)\sum_{\psi(\alpha)=0}\frac{1}{\psi'(\alpha)}\frac{\alpha}{(\alpha-1)^3}[z^n][K_{\alpha}(z)]\end{eqnarray}
(for $n=0$ one needs to add the term $C_4$).

For the first three terms in (\ref{calc}) we have, from (\ref{dec}) and lemmas \ref{l2} and \ref{l3}
\begin{eqnarray}\label{calc1}&&C_1 [z^n]\Big[\sqrt{\frac{1+z}{1-z}}\Big]+C_2 +C_3 [z^n]\Big[\frac{\sqrt{1+z}}{(1-z)^{\frac{3}{2}}}\Big]\\
&=&C_1  \frac{2}{\sqrt{2\pi}}\frac{1}{\sqrt{n}}+C_2+C_3\Big(\frac{4}{\sqrt{2\pi}} \sqrt{n}+\frac{1}{\sqrt{2\pi}}\frac{1}{\sqrt{n}}\Big)+O\Big(\frac{1}{n^{\frac{3}{2}}}\Big)\nonumber\\
&=&\frac{2}{\sqrt{2\pi}}\sqrt{n} +\frac{E(Y^2)-1}{2E(Y)}\nonumber\\&+&\frac{1}{\sqrt{2\pi}}\Big[\frac{1}{3}- \frac{E(Y^3)}{3E(Y)}+\frac{(E(Y^2)-1)^2}{2(E(Y))^2}\Big]\frac{1}{\sqrt{n}}+O\Big(\frac{1}{n^{\frac{3}{2}}}\Big).\nonumber\end{eqnarray}

The behavior of the last term in (\ref{calc}) is determined by Lemma \ref{l1}. If ($A_2$) holds, then by Lemma
\ref{not1}(ii) we have $\psi(-1)\neq 0$, so the $K_{-1}(z)$ is not among the summands, hence, by Lemma \ref{l1}(i) we have
\begin{equation}\label{calc3}2E(Y)\sum_{\psi(\alpha)=0}\frac{1}{\psi'(\alpha)}\frac{\alpha}{(\alpha-1)^3}[z^n][K_{\alpha}(z)]=O\Big(\frac{1}{n^{\frac{3}{2}}} \Big).\end{equation}
If ($A_2^c$) holds, then $K_{-1}(z)$ is among the summands, and by Lemma \ref{l1}(ii) it makes a higher order contribution so that
\begin{equation}\label{calc4}2E(Y)\sum_{\psi(\alpha)=0}\frac{1}{\psi'(\alpha)}\frac{\alpha}{(\alpha-1)^3}[z^n][K_{\alpha}(z)]=
-\frac{E(Y)}{4}\frac{(-1)^n}{\psi'(-1)}\frac{2}{\sqrt{2\pi}}\frac{1}{\sqrt{n}}+O\Big(\frac{1}{n^{\frac{3}{2}}} \Big).\end{equation}
Note that $\psi(-1)=0$ implies $h(-1)=-1$, and (from ($A_2^c$))
$$h'(-1)=\sum_{k=0}^\infty kp_k(-1)^k=-\sum_{k\;\;odd} kp_k=-E(Y),$$
so
\begin{equation}\label{pdd}(A_2^c)\;\;\Rightarrow\;\;\psi'(-1)=-\frac{1}{2}\phi'(-1)=-\frac{1}{2}[-2-2h(-1)+2h'(-1)]=E(Y),\end{equation}
hence (\ref{calc4}) gives
\begin{equation}\label{calc5}2E(Y)\sum_{\psi(\alpha)=0}\frac{1}{\psi'(\alpha)}\frac{\alpha}{(\alpha-1)^3}[z^n][K_{\alpha}(z)]=
\frac{(-1)^{n+1}}{4}\frac{2}{\sqrt{2\pi}}\frac{1}{\sqrt{n}}+O\Big(\frac{1}{n^{\frac{3}{2}}} \Big).\end{equation}

From (\ref{calc}),(\ref{calc1}) and (\ref{calc3}), we get the case $j=0$ of Theorem \ref{main}. From (\ref{calc}),(\ref{calc1}) and (\ref{calc5}), we get the case $j=0$ of Theorem \ref{main2}.

We now move on to the case $j\geq 1$. We note that, by Lemma \ref{gexp}, we can write
$$H_j(z)=\rho(z)^j H_0(z)+j\frac{1}{1-z},$$
so that
\begin{equation}\label{ze}E(X_n\;|\;X_0=j)=[z^n]H_j(z)=[z^n][\rho(z)^j H_0(z)]+j.\end{equation}
We shall find it easier to first estimate the coefficients $[z^n][z^j\rho(z)^j H_0(z)]$, and then use the obvious relation
\begin{equation}\label{or}[z^n][\rho(z)^j H_0(z)]=[z^{n+j}][z^j\rho(z)^j H_0(z)].\end{equation}

By the Binomial Theorem,
$$z^j\rho(z)^j=[1-\sqrt{1-z^2}]^j=\sum_{l=0}^j \binom{j}{l}(-1)^l (1-z^2)^{\frac{l}{2}},$$
hence, using (\ref{dec}),
\begin{eqnarray}\label{du}&&[z^n][z^j \rho(z)^j H_0(z)] = C_1 \sum_{l=0}^j \binom{j}{l}(-1)^l [z^n]\Big[\sqrt{\frac{1+z}{1-z}}(1-z^2)^{\frac{l}{2}}\Big]\\
&+&C_2 \sum_{l=0}^j \binom{j}{l}(-1)^l [z^n]\Big[\frac{(1-z^2)^{\frac{l}{2}}}{1-z}\Big]\nonumber\\&+&C_3 \sum_{l=0}^j \binom{j}{l}(-1)^l [z^n]\Big[\frac{\sqrt{1+z}}{(1-z)^{\frac{3}{2}}}(1-z^2)^{\frac{l}{2}}\Big]\nonumber\\
&+&C_4 \sum_{l=0}^j \binom{j}{l}(-1)^l [z^n]\Big[(1-z^2)^{\frac{l}{2}}\Big]\nonumber\\
&+&2E(Y)\sum_{\psi(\alpha)=0}\frac{1}{\psi'(\alpha)}\frac{\alpha}{(\alpha-1)^3}\sum_{l=0}^j \binom{j}{l}(-1)^l [z^n]\Big[K_{\alpha}(z)(1-z^2)^{\frac{l}{2}}\Big].\nonumber
\end{eqnarray}
Examining the first sum in (\ref{du}), we see from Lemma \ref{l2} that the only term in this sum making a contribution up to order $O\Big(\frac{1}{n^{\frac{3}{2}}}\Big)$
is the term corresponding to $l=0$, giving
\begin{eqnarray}\label{es1}&&\sum_{l=0}^j \binom{j}{l}(-1)^l  [z^n]\Big[\sqrt{\frac{1+z}{1-z}}(1-z^2)^{\frac{l}{2}}\Big]= [z^n]\Big[\sqrt{\frac{1+z}{1-z}}\Big]+O\Big(\frac{1}{n^{\frac{3}{2}}}\Big)\nonumber\\
&=&\frac{2}{\sqrt{2\pi}}\frac{1}{\sqrt{n}}+O\Big(\frac{1}{n^{\frac{3}{2}}}\Big).\end{eqnarray}

Similarly, for the second sum in (\ref{du}) we see from Lemma \ref{l5} that the only terms making a contribution up to order $O\Big(\frac{1}{n^{\frac{3}{2}}}\Big)$
are those corresponding to $l=0$ and $l=1$, giving
\begin{eqnarray}\label{es2}&&\sum_{l=0}^j \binom{j}{l}(-1)^l [z^n]\Big[ (1-z^2)^{\frac{l}{2}}\frac{1}{1-z}\Big]\\&=&[z^n]\Big[\frac{1}{1-z}-j\frac{(1-z^2)^{\frac{1}{2}}}{1-z}\Big]+O\Big(\frac{1}{n^{\frac{3}{2}}}\Big)
=1-j \frac{2}{\sqrt{2\pi}}\frac{1}{\sqrt{n}}+O\Big(\frac{1}{n^{\frac{3}{2}}}\Big).\nonumber\end{eqnarray}

For the third sum in (\ref{du}), we see from Lemma \ref{l3} that the terms making a contribution are those corresponding to $l=0,1,2$, giving
\begin{eqnarray}\label{es3}&&\sum_{l=0}^j \binom{j}{l}(-1)^l [z^n]\Big[(1-z^2)^{\frac{l}{2}}\frac{\sqrt{1+z}}{(1-z)^{\frac{3}{2}}}\Big]\\
&=&[z^n]\Big[\frac{\sqrt{1+z}}{(1-z)^{\frac{3}{2}}}\Big]-j[z^n]\Big[\frac{\sqrt{1+z}}{(1-z)^{\frac{3}{2}}}(1-z^2)^{\frac{1}{2}}\Big]\nonumber\\&+&
\frac{1}{2}j(j-1)[z^n]\Big[\frac{\sqrt{1+z}}{(1-z)^{\frac{3}{2}}} (1-z^2)\Big]+O\Big(\frac{1}{n^{\frac{3}{2}}}\Big)\nonumber\\
&=&\frac{2}{\sqrt{2\pi}} \sqrt{n}-j +\Big(\frac{1}{2}+j(j-1)\Big)\frac{1}{\sqrt{2\pi}}\frac{1}{\sqrt{n}} +O\Big(\frac{1}{n^{\frac{3}{2}}}\Big).\nonumber\end{eqnarray}

For the fourth sum in (\ref{du}), Lemma \ref{lll} implies that none of the terms make a contribution up to order $O\Big(\frac{1}{n^{\frac{3}{2}}}\Big)$, that is
\begin{equation}\label{es35}\sum_{l=0}^j \binom{j}{l}(-1)^l [z^n]\Big[(1-z^2)^{\frac{l}{2}}\Big]=O\Big(\frac{1}{n^{\frac{3}{2}}}\Big).\end{equation}

The behavior of the last sum in (\ref{du}) is determined by Lemma \ref{l1}. In case ($A_2$) holds, so that
$\alpha=-1$ is not a zero of $\psi$ (Lemma \ref{not1}(ii)) we have from Lemma \ref{l1}(i) that all terms in this sum are of order $O\Big(\frac{1}{n^{\frac{3}{2}}}\Big)$, so that
\begin{equation}\label{es41}\sum_{\psi(\alpha)=0}\frac{1}{\psi'(\alpha)}\frac{\alpha}{(\alpha-1)^3}\sum_{l=0}^j \binom{j}{l}(-1)^l [z^n]\Big[K_{\alpha}(z)(1-z^2)^{\frac{l}{2}}\Big]=O\Big(\frac{1}{n^{\frac{3}{2}}}\Big).\end{equation}
In the case that ($A_2^c$) holds, we have the term corresponding to $\alpha=-1$ in the sum, and in this case Lemma \ref{l1}(ii) tells us that the term
corresponding to $\alpha=-1$ and $l=0$ gives us a contribution, so that
\begin{eqnarray}\label{es42}&&\sum_{\psi(\alpha)=0}\frac{1}{\psi'(\alpha)}\frac{\alpha}{(\alpha-1)^3}\sum_{l=0}^j \binom{j}{l}(-1)^l [z^n]\Big[K_{\alpha}(z)(1-z^2)^{\frac{l}{2}}\Big]\\
&=&-\frac{1}{8}\frac{1}{\psi'(-1)}[z^n][K_{-1}(z)] +O\Big(\frac{1}{n^{\frac{3}{2}}}\Big)=\frac{1}{4E(Y)}\frac{1}{\sqrt{2\pi}}\frac{(-1)^{n+1}}{\sqrt{n}} +O\Big(\frac{1}{n^{\frac{3}{2}}}\Big),\nonumber\end{eqnarray}
where in the last step we have also used (\ref{pdd}).

Summing all these estimates, we have that, in the case ($A_2$) holds, (\ref{du})-(\ref{es41}) imply
\begin{eqnarray}\label{af1}&&[z^n][z^j \rho(z)^j H_0(z)] =\frac{2}{\sqrt{2\pi}} \sqrt{n}+(C_2-j)\\&+&\Big(C_1-jC_2+\frac{1}{4}+\frac{j(j-1)}{2}\Big)\frac{2}{\sqrt{2\pi}}\frac{1}{\sqrt{n}}
  +O\Big(\frac{1}{n^{\frac{3}{2}}}\Big),\nonumber\end{eqnarray}
while in the case that ($A_2^c$) holds, (\ref{du})-(\ref{es3}) and (\ref{es42}) imply
\begin{eqnarray}\label{af2}&&[z^n][z^j \rho(z)^j H_0(z)] =\frac{2}{\sqrt{2\pi}} \sqrt{n}+(C_2-j)\\&+&\Big(C_1-jC_2+\frac{1}{4}+\frac{j(j-1)}{2}+\frac{1}{8}(-1)^{n+1}\Big)\frac{2}{\sqrt{2\pi}}\frac{1}{\sqrt{n}}
  +O\Big(\frac{1}{n^{\frac{3}{2}}}\Big).\nonumber\end{eqnarray}

Assuming now that ($A_2$) holds, from (\ref{or}) and (\ref{af1}) we have
\begin{eqnarray}\label{ff}&&[z^n][\rho(z)^j H_0(z)]=\frac{2}{\sqrt{2\pi}} \sqrt{n+j}+C_2-j\\&+&\Big(C_1-jC_2+\frac{1}{4}+\frac{j(j-1)}{2}\Big)\frac{2}{\sqrt{2\pi}}\frac{1}{\sqrt{n+j}}
  +O\Big(\frac{1}{(n+j)^3} \Big),\nonumber
\end{eqnarray}
and in view of
\begin{eqnarray}\label{n1}&&\sqrt{n+j}=\sqrt{n}\sqrt{1+\frac{j}{n}}
=\sqrt{n}+\frac{j}{2}\frac{1}{\sqrt{n}}+O\Big(\frac{1}{n^{\frac{3}{2}}}\Big),\nonumber\\
&&\frac{1}{\sqrt{n+j}}=\frac{1}{\sqrt{n}}\frac{1}{\sqrt{1+\frac{j}{n}}}=
\frac{1}{\sqrt{n}}+O\Big(\frac{1}{n^{\frac{3}{2}}}\Big),\end{eqnarray}
(\ref{ff}) implies
\begin{eqnarray}\label{ff1}&&[z^n][\rho(z)^j H_0(z)]=\frac{2}{\sqrt{2\pi}}\sqrt{n} +(C_2-j)\\&+&\Big(C_1-jC_2+\frac{1}{4}+\frac{j^2}{2}\Big)\frac{2}{\sqrt{2\pi}}\frac{1}{\sqrt{n}}
  +O\Big(\frac{1}{n^{\frac{3}{2}}} \Big),\nonumber
\end{eqnarray}
So that together with (\ref{ze}) we have
\begin{eqnarray}\label{ff2}&&E(X_n\;|\;X_0=j)=\frac{2}{\sqrt{2\pi}}\sqrt{n} +C_2\\&+&\Big(C_1-jC_2+\frac{1}{4}+\frac{j^2}{2}\Big)\frac{2}{\sqrt{2\pi}}\frac{1}{\sqrt{n}}
  +O\Big(\frac{1}{n^{\frac{3}{2}}} \Big),\nonumber
\end{eqnarray}
and substituting the values of $C_1,C_2,C_3$ as calculated in Lemma \ref{l4}, we obtain the result of Theorem \ref{main}.

In the case that ($A_2^c$) holds, we carry out the same procedure, starting with (\ref{af2}), to obtain the result of Theorem \ref{main2}.

\section{The case $p_0=p_1=\frac{1}{2}$}

\label{special}

It remains to treat the special case given by (\ref{spc}). Since in this case $\phi(x)=1-x$, Lemma \ref{gexp} gives
$$z^{j+1}H_j(z)=\frac{z^{j+1}}{1-z}\frac{\rho(z)^{j+1}}{1-\rho(z)}+ j\frac{z^{j+1}}{1-z}$$
$$=\frac{1}{2}\Big[ \frac{\sqrt{1+z}}{(1-z)^{\frac{3}{2}}}+\frac{1}{1-z}\Big]\Big[1-\sqrt{1-z^2}\Big]^{j+1}+ j\frac{z^{j+1}}{1-z}$$
$$=\frac{1}{2}\Big[ \frac{\sqrt{1+z}}{(1-z)^{\frac{3}{2}}}+\frac{1}{1-z}\Big]\sum_{l=0}^{j+1} \binom{j+1}{l}(-1)^l(1-z^2)^{\frac{l}{2}}+ j\frac{z^{j+1}}{1-z}$$

So that, calculating as in section \ref{mainproof}, we have that for $n\geq j+1$:
$$[z^n][z^{j+1}H_j(z)]=\frac{1}{2} \sum_{l=0}^{j+1} \binom{j+1}{l}(-1)^l[z^n]\Big[\frac{\sqrt{1+z}}{(1-z)^{\frac{3}{2}}}(1-z^2)^{\frac{l}{2}}\Big]$$
$$+\frac{1}{2}\sum_{l=0}^{j+1} \binom{j+1}{l}(-1)^l[z^n]\Big[\frac{(1-z^2)^{\frac{l}{2}}}{1-z}\Big]+ j,$$
$$=\frac{1}{2}\Big[\frac{4}{\sqrt{2\pi}} \sqrt{n}+\frac{1}{\sqrt{2\pi}}\frac{1}{\sqrt{n}}-2(j+1)+j(j+1)\frac{2}{\sqrt{2\pi}}\frac{1}{\sqrt{n}}\Big] $$
$$+\frac{1}{2}\Big[ 1-(j+1) \frac{2}{\sqrt{2\pi}}\frac{1}{\sqrt{n}}\Big] +j +O\Big(\frac{1}{n^{\frac{3}{2}}} \Big) $$
$$=\frac{2}{\sqrt{2\pi}} \sqrt{n}-\frac{1}{2}+\Big(j^2-j-\frac{1}{2}\Big)\frac{1}{\sqrt{2\pi}}\frac{1}{\sqrt{n}} +O\Big(\frac{1}{n^{\frac{3}{2}}} \Big).$$

Therefore, using (\ref{n1}),
\begin{eqnarray}\label{resu}&&E(X_n\;|\; X_0=j)=[z^n][H_j(z)]=[z^{n+j+1}][z^{j+1}H_j(z)]\\
&=&\frac{2}{\sqrt{2\pi}} \sqrt{n+j+1}-\frac{1}{2}+\Big(2j^2-\frac{1}{2}\Big)\frac{1}{\sqrt{2\pi}}\frac{1}{\sqrt{n+j+1}} +O\Big(\frac{1}{n^{\frac{3}{2}}} \Big)\nonumber\\
&=&\frac{2}{\sqrt{2\pi}} \Big( \sqrt{n}+\frac{1}{2}(j+1)\frac{1}{\sqrt{n}}\Big)-\frac{1}{2}+\Big(j^2-\frac{1}{2}\Big)\frac{1}{\sqrt{2\pi}}\frac{1}{\sqrt{n}} +O\Big(\frac{1}{n^{\frac{3}{2}}} \Big)\nonumber\\
&=&\frac{2}{\sqrt{2\pi}}\sqrt{n}-\frac{1}{2}+\Big(j^2+j+\frac{1}{2}\Big)\frac{1}{\sqrt{2\pi}}\frac{1}{\sqrt{n}} +O\Big(\frac{1}{n^{\frac{3}{2}}} \Big)\nonumber.\end{eqnarray}
Since in the case under consideration we have $E(Y)=E(Y^2)=E(Y^3)=\frac{1}{2}$, it can be checked that the result (\ref{resu}) is exactly that
given by Theorem \ref{main}.

\end{document}